\documentclass{amsart}
\usepackage[utf8]{inputenc}
\usepackage{amssymb}
\usepackage{amsmath}
\usepackage{array}
\usepackage{hyperref}
\usepackage{a4wide}
\usepackage{xypic}

 \DeclareMathOperator\Aut{{\rm Aut}}

\newcommand\liegr{\sf}

\newcommand{\SU}[1]{\mbox{${\liegr SU}(#1)$}}

\newcommand{\U}[1]{\mbox{${\liegr U}(#1)$}}
\newcommand{\SP}[1]{\mbox{${\liegr Sp}(#1)$}}
\newcommand{\SO}[1]{\mbox{${\liegr SO}(#1)$}}
\newcommand{\SOxO}[2]{\mbox{${\liegr S(O}(#1)\times{\liegr O}(#2))$}}
\newcommand{\OG}[1]{\mbox{${\liegr O}(#1)$}}
\newcommand{\Spin}[1]{\mbox{${\liegr Spin}(#1)$}}

\newcommand{\G}{\mbox{$\liegr{G}_2$}}

\newcommand{\F}{\mbox{$\liegr{F}_4$}}

\newcommand{\Lg}{\mbox{$\mathfrak g$}}

\newcommand{\Lh}{\mbox{$\mathfrak h$}}

\newcommand\fieldsetc{\mathbb}

\renewcommand{\H}{\fieldsetc{H}}
\newcommand{\Z}{\fieldsetc{Z}}
\newcommand{\R}{\fieldsetc{R}}
\newcommand{\C}{\fieldsetc{C}}
\newcommand{\Q}{\fieldsetc{H}}
\newcommand{\Ca}{\fieldsetc{O}}
\newcommand{\K}{\fieldsetc{K}}

\newcommand{\Cdiag}[3]
{\begin{xy}
\POS (0,0)*\cir<3pt>{} ="a", \POS (0,3)*{\scriptstyle #1},
\POS (10,0)*\cir<3pt>{} ="b", \POS (10,3)*{\scriptstyle #2},
\POS (20,0)*\cir<3pt>{} ="c", \POS (20,3)*{\scriptstyle #3},
\POS "a" \ar@{-}^<<<<{} "b",
\POS "b" \ar@{=}^<<<<{} "c",
\end{xy}}

\newcommand{\ACdiag}[3]
{\begin{xy}
\POS (0,0)*\cir<3pt>{} ="a", \POS (0,3)*{\scriptstyle #1},
\POS (10,0)*\cir<3pt>{} ="b", \POS (10,3)*{\scriptstyle #2},
\POS (20,0)*\cir<3pt>{} ="c", \POS (20,3)*{\scriptstyle #3},
\POS "b" \ar@{=}^<<<<{} "c",
\end{xy}}

\newcommand{\AAAdiag}[3]
{\begin{xy}
\POS (0,0)*\cir<3pt>{} ="a", \POS (0,3)*{\scriptstyle #1},
\POS (10,0)*\cir<3pt>{} ="b", \POS (10,3)*{\scriptstyle #2},
\POS (20,0)*\cir<3pt>{} ="c", \POS (20,3)*{\scriptstyle #3},
\end{xy}}

\newcommand{\Atildediag}[2]
{\begin{xy}
\POS (0,0)*\cir<3pt>{} ="a", \POS (0,3)*{\scriptstyle #1},
\POS (10,0)*\cir<3pt>{} ="b", \POS (10,3)*{\scriptstyle #2},
\POS "a" \ar@{-}^<<<<{\;\;\infty} "b",
\end{xy}}

\newtheorem{thm}{Theorem}[section]
\newtheorem{cor}[thm]{Corollary}

\newtheorem{lem}[thm]{Lemma}
\newtheorem{sch}[thm]{Scholium}
\theoremstyle{remark}
\newtheorem{rmk}[thm]{Remark}

\newtheorem{quest}[thm]{Questions}

\begin{document}
\setcounter{tocdepth}{1}

%
%
%

\pagenumbering{arabic}
\pagestyle{plain}

\title{Some remarks on polar actions}

\author[C.~Gorodski]{Claudio Gorodski}
\author[A.~Kollross]{Andreas Kollross}

\address{Instituto de Matem\'atica e Estat\'\i stica, Universidade de
S\~ao Paulo, Rua do Mat\~ao, 1010, S\~ao Paulo, SP 05508-090, Brazil}

\email{gorodski@ime.usp.br}

\address{Universität Stuttgart, Institut für Geometrie und Topologie,
Pfaffenwaldring~57, 70569~Stuttgart, Germany}

\email{kollross@mathematik.uni-stuttgart.de}

\thanks{The first author was partially supported by the
CNPq grant 303038/2013-6 and the FAPESP project 2011/21362-2.}

\date{\today}



\begin{abstract}
We classify infinitesimally polar actions on compact Riemannian symmetric spaces of rank one. We also prove that every polar action on one of those spaces has
the same orbits as an asystatic action.

\end{abstract}

\maketitle


\section{Introduction and Results}

A Riemannian orbifold is a metric space which is locally modeled
on quotients of Riemannian manifolds by finite groups of isometries.
It has been shown by Lytchak and Thorbergsson~\cite{L-T} that the orbit
space of a proper and isometric action of a Lie group on a Riemannian
manifold, with the quotient metric space structure, is a Riemannian
orbifold if and only if the action is \emph{infinitesimally polar},
which means that all of its slice representations are polar.
Recall that a proper and isometric action of a Lie group on a
complete Riemannian manifold is called \emph{polar}
if there exists a connected complete isometrically immersed
submanifold, called a \emph{section}, meeting all orbits and always
orthogonally. A section is automatically totally geodesic, and if it is
flat in the induced metric then the action is called
\emph{hyperpolar}~\cite{PT2,BCO}.

Infinitesimally polar actions of connected compact Lie groups
on Euclidean spheres have been classified in~\cite{G-L3}.
It is obvious that infinitesimally polar
actions on real projective spaces are exactly those actions
induced from infinitesimally polar actions on spheres.
In the present paper, we classify infinitesimally polar actions
on the remaining compact rank one symmetric spaces.
Note that general
polar actions are infinitesimally polar (\cite[Thm~4.6]{PT2}
or~\cite[Prop.~3.2.2]{BCO}), and polar actions on
compact rank one symmetric spaces have been classified
by Podest\`a and Thorbergsson in~\cite{P-Th1}
(although they have overlooked one case, cf.~subsection~\ref{g2so3}).
Our result is as follows.

\begin{thm}\label{main}
Assume a compact connected Lie group acts isometrically and
effectively on a compact rank one symmetric space $M$.
\begin{enumerate}
\item[(a)]
If $M$ is a complex projective space
$\C P^m$, then the action is infinitesimally polar
if and only if it is polar or it is orbit equivalent
to the action induced from one of the following
representations~$\rho$ of~$G$:
 \setlength{\extrarowheight}{0.1cm}
\[ \begin{array}{|c|c|c|c|c|}
\hline
G & \rho & \text{Conditions} & m & \text{Orbit space}\\
\hline
\SO2\times\Spin9 & \R^2\otimes_{\mathbb R}\R^{16} & - & 15 & S^2_{+++}(\frac12)\\
\U 2\times\SP n  & \C^2\otimes_{\mathbb C}\C^{2n} & n\geq2& 4n-1 & S^2_{+++}(\frac12)\\
\U n & \C^n\oplus\C^n & n\geq2 & 2n-1 & S^3_+(\frac12)\\
\liegr T^ 2 \times \SU n & \C^n\oplus\C^n & n\geq2 & 2n-1 & S^2_{++}(\frac12)\\
\U1\times \SP n& \C^{2n}\oplus\C^{2n} & n\geq2 & 4n-1 & S^4_{++}(\frac12)\\
\liegr T^2\times \SP n & \C^{2n} \oplus \C^{2n} & n\geq2& 4n-1&
S^3_{+++}(\frac12)\\
\hline
\end{array}
\]
\begin{center}
\hfill\\
\textsc{Table~1}\hfill\\
\hfill\\
\end{center}

\item[(b)]
If $M$ is a quaternionic projective space
$\Q P^m$ with $m>1$,
then the action is infinitesimally polar
if and only if it is polar or it is orbit equivalent to
the action induced from the following
representation~$\rho$ of~$G$:\hfill\\[1em]
 \setlength{\extrarowheight}{0.1cm}
\begin{minipage}{\textwidth}
\[ \begin{array}{|c|c|c|c|c|}
\hline
G & \rho & \text{Conditions} & m & \text{Orbit space}\\
\hline
\SP n\times\SP1  & \R^{4n}\oplus\R^{4n} & n\geq2 & 2n-1 & S^3_{++}(\frac12)\\
\hline
\end{array}
\]
\begin{center}
\hfill\\
\textsc{Table~2}\hfill\\
\hfill\\
\end{center}
\end{minipage}

\item[(c)]
If $M$ is the Cayley projective plane $\Ca P^2$, then
the action is infinitesimally polar if and only if it is polar.
In this case the action is conjugate to one given by the following
subgroups $G$ of $\F$:\hfill\\[1em]
\begin{minipage}{\textwidth}
\setlength{\extrarowheight}{0.35cm}
\[ \begin{array}{|c|c|c|c|}\hline
G & \text{Cohom}  & \text{Type} & \text{Multiplicities}\\[1.1ex]\hline
\Spin9 & 1 & {\liegr\tilde A_1}& \Atildediag{15}7\\[1.1ex]
\renewcommand{\arraystretch}{.6}
\begin{array}{c}\SP3\cdot\SP1\\\SP3\cdot\U1  \\\SP3\end{array}&  1 &{\liegr\tilde A_1} & \Atildediag74\\[1.1ex]
\Spin8 & 2 &{\liegr A_1\times\liegr A_1\times\liegr A_1}&\AAAdiag777\\[1.1ex]
\Spin7\cdot\SO2 & 2 &{\liegr A_1\times\liegr C_2}&\ACdiag761\\[1.1ex]
\SU4\cdot\SU2 & 2 & {\liegr A_1\times\liegr C_2} &\ACdiag752\\[1.1ex]
\SU3\cdot\SU3 & 2 & {\liegr C_3}&\Cdiag223\\[1.1ex]
\SO3 \cdot \G & 2 & {\liegr C_3}&\Cdiag115\\[1.1ex]\hline
\end{array} \]
\begin{center}
\hfill\\
\textsc{Table~3}\hfill\\
\hfill\\\end{center}
\end{minipage}
\end{enumerate}

The notation $S^n_+(r)$, $S^n_{++}(r)$ and $S^n_{+++}(r)$ in Tables~1 and~2 
stands for the quotient of the $n$-sphere of radius~$r$ by a group 
generated by $1$, $2$, resp.~$3$, commuting reflections.
In Table~3 we indicate the type of the Coxeter group acting on
the universal covering of the section and its multiplicities
(these are the dimensions of the unit spheres in the normal space,
at a generic point, to a sub-principal stratum of $M$ ---
that is, a stratum in~$M$
projecting to a codimension one stratum in the orbit space).
\end{thm}

Following Thurston~\cite{Thu}, we call a Riemannian orbifold
\emph{good} if it is globally isometric to the quotient
of a Riemannian manifold by a discrete group of isometries.

\begin{cor}
An infinitesimally polar action on a compact rank one symmetric
space has a good Riemannian orbifold as a quotient.
\end{cor}

It is relevant to mention
the fundamental work of Dadok, who classified
polar representations~\cite{D}.
It follows from his result that a polar representation
of a connected compact Lie group is orbit equivalent to
the isotropy representation of a Riemannian symmetric space.
Herein we say that two isometric Lie group actions on Riemannian manifolds
 are \emph{orbit equivalent} if they have the same orbits after a suitable
 isometric identification of the manifolds.

Recall that a homogeneous manifold is called \emph{asystatic}
if sufficiently close points in the manifold have different isotropy groups;
equivalent definitions are: that the manifold has no nonzero invariant
vector field; or the isotropy representation has no nonzero
fixed vectors; or that the normalizer of the isotropy group is a discrete
extension thereof (this concept can be traced back to
S.~Lie~\cite[p.501]{Lie}; see also~\cite{P-Th2}).
Finally, a proper action of a Lie group on a smooth
manifold is called \emph{asystatic} if one (and hence all)
principal orbits are asystatic homogeneous manifolds~\cite{A-A}.
The relevance to us is that asystatic actions
are automatically polar
with respect to any invariant metric: a section is given by a
connected component of the set of fixed points which
contains regular points of the action
(therefore, in~\cite{PT} they are called
\emph{$G$-manifolds with canonical sections}).
Note also that the asystatic property actually depends on the
group and not only on
its orbits. It follows from our discussion below that:

\begin{sch}\label{asys}
Every polar action of a compact connected Lie group
on a compact rank one symmetric space
is orbit equivalent to an asystatic action.
\end{sch}

More precisely, each group with a polar action on such a space
admits an extension, by a finite group,
which acts with the same orbits and
asystatically.
The action of the original group will thus
be polar with respect to any Riemannian metric invariant
under the enlarged group.

Note that the sections of asystatic actions are automatically
properly embedded submanifolds, but
this is not true for general polar actions, see e.g.~\cite[p.~47]{GTh4}. 
It is an interesting open question, communicated to us by W. Ziller,
to decide whether sections of polar actions can admit self-intersections
(in the case of cohomogeneity one actions it is known that they 
cannot~\cite[Thm.~6.1]{aa}).
Since sections of polar actions on compact irreducible symmetric spaces 
of higher rank are known to be properly embedded~\cite{HPTT,KL}, one can ask:

\begin{quest}
Can the result in Scholium~\ref{asys} be extended to the case
of compact irreducible symmetric spaces of higher rank? 
What is the most general class of polar actions on complete Riemannian
manifolds for which it is true?
\end{quest}

The authors would like to express their gratitude to Alexander
Lytchak for very informative discussions and his kind hospitality during 
their stay at the University of Cologne. 

\section{Actions on classical projective spaces}\label{classical}

\subsection{Infinitesimally polar actions}
We view the complex and quaternionic projective spaces as
quotients of unit spheres under the corresponding Hopf actions
so that their quotient Riemannian metrics have sectional curvatures
lying between $1$ and $4$. Moreover, a $2$-plane with sectional curvature
equal to $1$ must be totally real.

Assume an isometric Lie group action on $\C P^m$ is given
by specifying a closed subgroup $G$ of $\SU{m+1}$.
This action lifts, via the Hopf fibration, to an isometric
action of the group $\tilde G:=G\times\U1$ on $S^{2m+1}$ with the same orbit space,
where we can view $S^{2m+1}$ as the unit sphere in $\C^{m+1}$
and $\U1$ acts by multiplication by unit complex numbers.
It follows that the $G$-action on $\C P^m$ is
infinitesimally polar if and only the
$\tilde G$-action on $S^{2m+1}$ is
infinitesimally polar.
Moreover, it follows from O'Neill's formula for Riemannian submersions
that an isometric action on the unit sphere is polar
if and only if it has cohomogeneity one
or the principal stratum of the orbit space
has constant sectional curvature~$1$, cf.~\cite[Introd.]{G-L2}.
Therefore another application of O'Neill's formula
gives that polarity of the $\tilde G$-action on $S^{2m+1}$ implies
polarity of the $G$-action on $\C P^m$ with totally real sections.
Conversely, a polar action on $\C P^m$ must have totally real
sections (see~\cite[Thm.~1.1]{P-Th2} or~\cite[Prop.~9.1]{Ly2}) and thus
the lifted action on $S^{2m+1}$ is polar.

Similarly, let an isometric Lie group action $\Q P^m$ be given
by specifying a closed subgroup
$G$ of $\SP{m+1}$. This action lifts, via the Hopf fibration, to an
isometric action of $\tilde G:=G\times\SP1$ on $S^{4m+3}$
with the same orbit space,
where we can view $S^{4m+3}$ as the unit sphere in $\Q^{m+1}$
and $\SP1$ acts by (right) multiplication by unit quaternion numbers.
It follows as above that the $G$-action on $\Q P^m$
($m>1$, cf.~\cite[p.~161]{P-Th1}) is
non-polar and infinitesimally polar if and only the
$\tilde G$-action on $S^{4m+3}$ is non-polar and
infinitesimally polar.

Therefore Tables~1 and~2 follow
from~\cite[Th.~1.3]{G-L3} and \cite[Table~II]{Str}.

\subsection{Asystatic actions}
Straume proved that a polar representation is orbit equivalent to
an asystatic representation, using the following argument;
see~\cite[Completion of proof of Thm.~1.3, pp.~11-12]{Str},
although he does not use the word ``asystatic'' explicitly.
By~\cite{D}, one may replace the representation by one with the same orbits
which is the isotropy representation
of a symmetric space. In the irreducible case, one checks case-by-case
that this representation is already asystatic unless it is of
Hermitian type, in which case it becomes asystatic after adjunction
of an extra element to the group, namely, complex conjugation, without
changing its orbits.
In the reducible case, the representation splits as the direct
product of irreducible
isotropy representations of symmetric spaces, and it also becomes
asystatic by adjoining one single extra element that acts as
complex conjugation.

As a consequence of Straume's result, also the polar actions on
classical projective
spaces are asystatic. In fact, a polar $G$-action,
where $G \subset \SU{m+1}$, on a
complex projective space $M = \C P^m$ lifts to a polar $\tilde G$-action
on the corresponding
sphere $\tilde M = S^{2m+1}$, which can be replaced by a group
acting with the same orbits via the isotropy representation of an
Hermitian symmetric space, see
\cite[Theorem~3.1]{P-Th1}.
In its turn, the latter group can by the above
be enlarged to a group $\tilde K$ acting asystatically
on $\tilde M$ with the same orbits, by adjunction of complex conjugation,
and which induces an action of a group $K$
on $M$ orbit equivalent to the original $G$-action.
If $\tilde p\in\tilde M$, $p\in M$ are $\tilde K$-regular,
resp. $K$-regular
points, where $\tilde p$ lies above~$p$ and
$\tilde H=\tilde K_{\tilde p}$, $H=K_p$ are
the associated principal isotropy groups,
then the projection $\tilde K\to K$ induces an
isomorphism $\tilde H\cong H$ with respect to which the isomorphism
$T_p(\tilde K\cdot \tilde p)\cap\mathcal H_{\tilde p} \cong T_p(K \cdot p)$
(induced by the projection $\tilde M\to M$, where $\mathcal H$ denotes the
horizontal distribution) is equivariant.
Hence there exist no $H$-fixed directions
in $T_p(K \cdot p)$, i.e.~$K$ acts asystatically on~$M$.

For actions on quaternionic projective space, we may use an analogous method,
however, the argument has to be somewhat refined. Assume now $G \subset \SP{m+1}$
acts polarly on the quaternionic projective space $M = \Q P^m$.
Then the action lifts to a polar
$\tilde G$-action on the corresponding sphere $\tilde M = S^{4m+3}$.
This action can be replaced by a group acting
on $\H^{m+1} = \H^{m_1} \oplus \dots \oplus \H^{m_r}$ with the same
orbits via the  isotropy representation of the product of $r$~quaternion K\"ahler symmetric spaces $Q_1, \dots, Q_r$, where $Q_1,\ldots,Q_{r-1}$ are of rank one
and $Q_r$ can be of arbitrary rank, cf.~\cite[Theorem~4.1]{P-Th1}.
The isotropy representations of quaternion Kähler symmetric spaces
are known to be asystatic in all but one case:
there is exactly one case of quaternion K\"ahler symmetric space
which is also a Hermitian symmetric space, namely, the Grassmann manifold
$\SU{n+2}/\mathsf{S}(\U n\times\U2)$ of complex $2$-planes in $\C^{n+2}$,
in which we need to pass to a $\Z_2$-extension to make it asystatic.
However, as remarked in~\cite{P-Th1} after Proposition~2A.2,
this action does not descend to~$M$ if $r \ge 2$.
To remedy this, one takes the subgroup
$K := \SP{m_1} \times \dots \times \SP{m_{r-1}} \times H_r \times \SP1$
which acts on $\H^{m+1}$ with the same orbits, where
the $\SP1$-factor acts by right quaternionic multiplication
and $(K_r=H_r\times\SP1,\H^{m_r}\otimes_{\mathbb H}\H)$
corresponds to the isotropy representation of
$Q_r$. The $K$-action on $\tilde M$,
\begin{equation*}\label{product}
 (K=\SP{m_1}\times\cdots\times\SP{m_{r-1}}\times H_r\times\SP1,
 \H^{m+1} = \H^{m_1} \oplus \dots \oplus \H^{m_r}),
\end{equation*}
descends to an
action on $M$ which is orbit equivalent to the original action;
let us show that this $K$-action is asystatic.
Up to conjugation of $H_r$ in $\SP{m_r}$, we may
assume that there exists a regular point $p\in\H^{m+1}$ of the form
$(p_1,\ldots,p_r)$, where each $p_\ell$ is the first element
of the canonical basis of $\H^{m_\ell}$. Then the corresponding
principal isotropy group $K_p$ is isomorphic to
\[
\SP{m_1-1}\times\cdots\times\SP{m_{r-1}}\times (K_r)_{p_r},
\]
where $(K_r)_{p_r}$ is the principal isotropy group
of $(K_r,\H^{m_r}\otimes_{\mathbb H}\H)$.
The group $K_p$ acts on
$\H^{m+1}=\H^{m_1}\oplus\cdots\oplus\H^{m_r}$ by
\begin{align*}
(&A_1, \dots, A_{r-1}, (A_r,q)) \cdot (x_1,\ldots,x_r)=\\
&=((qx_1^0q^{-1},A_1x_1'q^{-1}),\dots,(qx_{r-1}^0q^{-1},A_{r-1}x_{r-1}'q^{-1}),
A_rx_rq^{-1}),
\end{align*}
where $A_\ell \in \SP{m_\ell-1}$ for $\ell<r$, $(A_r,q) \in (K_r)_{p_r} \subset H_r \times \SP1$, $x_\ell=(x_\ell^0,x_\ell') \in \H\oplus\H^{m_\ell-1}$ for $\ell<r$.
By polarity, the tangent space
to the $K$-principal orbit through $p$ is the direct sum of the tangent spaces
of the $K$-orbits through the $p_\ell$.
First assume $Q_r$ is not a
Hermitian symmetric space, namely, it is not the
Grassmann manifold of complex $2$-planes
in~$\C^{m_r+2}$. Then,
by~\cite{Str}, the $K_r$-action on $\H^{m_r}$ is asystatic.
Therefore the $(K_r)_{p_r}$-action on $T_{p_r}(K_r \cdot p_r)$ has no fixed
directions; note that the $(K_r)_{p_r}$-action on $\H^{m_r}$ is the
effectivized $K_p$-action on that space. In order
to see that $K_p$ has no fixed tangent directions in the other
summands $\H^{m_1},\ldots,\H^{m_{r-1}}$, we argue as follows.
By~\cite{Teb}, the section of $(K_r,\H^{m_r})$ through $p_r$ is
totally real, namely, it is orthogonal to its image under
right multiplication by an imaginary unit quaternion. By asystaticity,
the only $(K_r)_{p_r}$-fixed direction in the quaternionic span of $p_r$
is $\R\,p_r$. Since the effective group acting on $\H^{m_\ell}$ for
$\ell<r$ is the full $\SP{m_\ell}\cdot\SP1$, this behavior is reproduced in $\H^{m_\ell}$,
namely, the only $K_p$-fixed direction in the quaternionic span of $p_\ell$
is $\R\,p_\ell$. Since the $K_p$-action on $\H^{m_\ell}$ contains the subgroup $\{1\}\times\SP{m_\ell-1}\subset\SP{m_\ell}$ which does not fix non-zero vectors
in $\{0\}\oplus\H^{m_\ell-1}$, this already shows that there are no further fixed directions in $\H^{m_\ell}$.

In case $Q_r$ is the complex Grassmannian of complex $2$-planes
in~$\C^{m_r+2}$,
we argue analogously by replacing
the group $K_r$ with the group $\hat K_r$
generated by $K_r$ and the element $\sigma_r$
that acts by complex conjugation on~$\H^{m_r}=\C^{m_r}\oplus\C^{m_r}j$, and trivially
on~$\H^{m_\ell}$ for $\ell<r$.
Note that $\sigma_r$ is given on $\H^{m_r}$ by
$L_j \circ R_j^{-1}$ (left and right multiplication),
which shows that $\sigma_r \in \SP{m_r} \times \SP1$.
The resulting group $\hat K$ has the same orbits as $K$
and acts asystatically on $\H^{m+1}$.
It is generated by $K$ and an element~$\sigma$
that acts as $L_j \circ R_j^{-1}$
on each $\H^{m_\ell}$ (again due to the fact that the effective
group on $\H^{m_\ell}$ for $\ell<r$ is $\SP{m_\ell}\cdot\SP1$),
which fixes $p$ and preserves quaternionic lines
in $\H^{m+1}$, so $\hat K$ induces an asystatic action on $\H P^m$.

\section{Actions on the Cayley projective plane}

We first prove a simple but useful criterion for asystaticity.

\begin{lem}\label{crit-asys}
Let $G$ act properly and isometrically on a complete Riemannian
manifold~$M$. Assume there exists a point $q\in M$ and a
principal isotropy group $H$ which fixes~$q$ and acts on
the tangent space $T_qM$ with fixed point set of dimension
equal to the cohomogeneity of the $G$-action on $M$. Then the
$G$-action on $M$ is asystatic.
\end{lem}

\begin{proof} We may assume the action is non-transitive
and the point $q$ is not regular. There is a non-zero
normal vector~$v\in\nu_q(Gq)$
which is regular for the slice representation of $G_q$
and fixed under $H$. There is a sequence of regular points
$(p_n:=\exp_q(t_nv))$, where $t_n\to0$ as $n\to\infty$,
which converges to $q$.
The isotropy representations $(H=G_{p_n},T_{p_n}(G \cdot p_n))$ of the principal
orbits $G \cdot p_n$ are all equivalent one to the other, for all~$n$.
By continuity, the dimension of the fixed point set of $H$ in
$T_{p_n}M$ is not larger than the dimension of the fixed point
set of $H$ in $T_qM$, hence it is equal to the cohomogeneity of the
$G$-action on $M$. \end{proof}

We will also use the following lemma~\cite[Th.~6]{K-P}.

\begin{lem}\label{polar-min}
Let $\rho:G\to \OG{V}$ be a faithful irreducible representation of a
compact connected Lie group $G$ of cohomogeneity at least~$2$.
Assume the restriction of $\rho$ to a non-trivial closed connected subgroup $H$ is polar.
Then the $G$- and $H$-actions on $V$ are orbit equivalent.
\end{lem}

We will throughout use Dynkin's tables of maximal
connected closed subgroups of compact
Lie groups~\cite{Dyn0}. The identity component of the isometry group of
$M=\Ca P^2$ is the exceptional Lie group $\F$, and its isotropy group
at a fixed basepoint is $\Spin9$. The maximal connected closed
subgroups of $\F$ are, up to
conjugacy~\cite[Table~12 and Thm.~14.1]{Dyn0}:
\[ \Spin9,\quad \SP3\cdot\SP1,\quad \SU3\cdot\SU3,\quad \G\cdot\liegr{A}_1^8,\quad \liegr{A}_1^{156}. \]
(By the way, the group $\widetilde{\SU2}\cdot\SU4$ is not maximal
and occurs in Dynkin's tables because of a mistake, see e.g.~\cite{G-R2}.)
Here $\liegr{A}_1$ denotes a simple group of rank~$1$ and the upper
index refers to its Dynkin index as a subgroup of $\F$, as defined
in~\cite{Dyn0}.

\subsection{$\Spin9$ and its subgroups}\label{spin9}

The action of $\Spin9$ is its isotropy action on the homogeneous
space $M=\F/\Spin9$.
Since $M$ is a symmetric space of rank $1$, this action has cohomogeneity
$1$ and it is thus hyperpolar. The principal orbits are distance
spheres $S^{15}=\Spin9/\Spin7$ centered at the basepoint, and their
isotropy representation decomposes as
$\R^7\oplus\R^8$, where $\R^7$ denotes the
vector representation of $\Spin7$ and $\R^8$ denotes its spin representation.
It follows that $\Spin9$ acts asystatically on $M$.

We proceed to consider the maximal connected closed subgroups of
$\Spin9$ of rank greater than one (cf.~subsection~\ref{rk1}), which are
\[
\Spin8,\quad \Spin7\cdot\SO2,\quad \Spin6\cdot\Spin3,\quad \Spin5\cdot\Spin4
\]
and an embedding of $\SP1\cdot\SP1$ coming from the representation
of $\SO4$ on the space of real traceless symmetric $4\times4$ matrices.

\subsubsection{$\Spin8$}
The slice representation of $\Spin8$ at the basepoint decomposes as
the direct sum of the half-spin representations, so the principal isotropy
of its action on $M$ is $\G$. Now the isotropy representation of a
principal orbit $\Spin8/\G$ is easily seen to be $\R^7\oplus\R^7$.
It follows that $\Spin8$ acts asystatically on~$M$.

It remains to determine whether subgroups of $\Spin8$ can act infinitesimally polar. To this end, we consider the maximal connected subgroups in $\Spin8$ of rank greater than one.  Since outer automorphisms of $\Spin8$ are restrictions of inner automorphisms of~$\F$, it suffices to consider maximal connected subgroups up to arbitrary automorphisms of $\Spin8$, and
they are (cf.~\cite[Prop.~3.3]{K}):
\[
\Spin7,\quad \Spin6\cdot\SO2,\quad \Spin5\cdot\Spin3,\quad \Spin4\cdot\Spin4,\quad \pi^{-1}(\rm Ad\,{\SU3}),
\]
where $\pi \colon \Spin8 \to \SO8$ is a covering map and where $\rm Ad\,{\SU3}$ is the group given by the $8$-dimensional irreducible representation of~$\SU3$.

We do not need to discuss $\Spin7$, $\Spin6\cdot\SO2$, $\Spin5\cdot\Spin3$,
or $\Spin4\cdot \Spin4$ or any of their closed subgroups now,
since each one
of them is contained in one of the subgroups $\Spin7\cdot\SO2$,
$\Spin6\cdot\Spin3$, $\Spin5\cdot\Spin4$ of $\Spin9$, and the latter will
be treated below. We do not need to discuss the group
$\pi^{-1}(\rm Ad\,{\SU3})$, either, since it is contained in
$\SU3\cdot\SU3$, see \cite[p.~195]{Dyn0}.

\subsubsection{$\Spin7\cdot\SO2$ and $\Spin6\cdot\Spin3$}\label{7263}
For the next two group actions, it will be convenient to use
the notion of Weyl involution of a compact Lie group with
respect to a maximal torus~\cite[\S2]{On}.
Recall that a regular subalgebra of a compact semisimple
Lie algebra $\Lg$ is a subalgebra which is normalized by a maximal torus
of $\mathrm{Int}(\Lg)$. Subalgebras of maximal rank are obviously
regular.
It follows easily that for any regular subalgebra, there exists
a Weyl involution preserving the
subalgebra and restricting to its Weyl involution.
We apply these ideas to the case of the group $\F$,
where any Weyl involution is an inner automorphism.

The slice representation of $\Spin7\cdot\SO2$ at the basepoint~$p$
is $\R^8\otimes_{\mathbb R}\R^2$, so the principal isotropy group of its
action on $M$ is $\SU3\rtimes\Z_2$, where
$\SU3\subset\G\subset\Spin7$ and $\Z_2$ is diagonally embedded
in $\Spin7\cdot\SO2$. One computes the
isotropy representation of a principal orbit to be
$2\C^3\oplus2\R$, where the $\Z_2$-factor acts trivially
on $2\R$ (one summand corresponds to $\mathfrak{so}(2)$ and the other
to the center of $\mathfrak{u}(3)\subset\mathfrak{so}(7)$).
This shows the $\Spin7\cdot\SO2$-principal orbit is not asystatic.
To deal with this, we enlarge $\Spin7\cdot\SO2$ by adjoining
$w\in\F$, where $\psi=\mathrm{Ad}_w$
is a Weyl involution of~$\F$ relative to a maximal torus
contained in $\Spin7\cdot\SO2$. Note that $\psi$ preserves $\Spin9$,
which does not have outer automorphisms, so we may
take $w\in\Spin9$. Now
$\psi$ induces an isometry of $M$ that
maps $\Spin7\cdot\SO2$-orbits to $\Spin7\cdot\SO2$-orbits,
and thus it induces an isometry of the corresponding orbit space.
Since $\psi$ fixes
the basepoint $p$, it also induces an isometry of the orbit space of the
slice representation of~$\Spin7\cdot\SO2$ at~$p$. The latter
is orbit equivalent to
$(\SO8\times\SO2,\R^8\otimes_{\mathbb R}\R^2)$~\cite{D,EH}
and its orbit space is thus isometric to the cone over the
interval of length~$\pi/4$.
The only non-trivial isometry interchanges the endpoints of the interval,
but such an isometry cannot be induced by $\psi$ because the endpoints
parametrize singular orbits of different dimensions,
namely, $8$ and $13$~\cite[p.~436]{HPT}. It follows that $\psi$
preserves the orbits in $T_pM$, and hence preserves the orbits in $M$
by using the exponential map at~$p$.
Now the enlarged group has the same orbits in $M$ and
its principal isotropy group contains $w$. Since $\psi$
acts as minus identity on the Lie algebra of the maximal torus
of $\Spin7\cdot\SO2$, this shows that the enlarged group acts asystatically.

Consider now $\Spin6\cdot\Spin3=\SU4\cdot\SU2$; this case is very
similar to the previous one. Its slice representation
at~$p$ is $\C^4\otimes_{\mathbb C}\C^2$, so the principal isotropy
group of its action on $M$ is~$H\cong\U2$, and one computes
that the isotropy representation of a principal orbit
has fixed point set of dimension $2$, namely, the
Killing orthogonal of $\Lh\cong\mathfrak{u}(2)$ in the
Lie algebra of the maximal torus of $\SU4\cdot\SU2$.
We enlarge $\SU4\cdot\SU2$ by adjoining $w\in\Spin9$
where $\psi=\mathrm{Ad}_w$ is a Weyl involution
of~$\F$ relative to a maximal torus contained in $\SU4\cdot\SU2$.
In order to see that the enlarged group has the same orbits,
it suffices, as above, to note that $\psi$ fixes~$p$ and
that it cannot induce a non-trivial isometry of the
orbit space of the slice representation at~$p$. Indeed the latter
is orbit equivalent to $(\mathsf{S}(\U4\times\U2),\C^4\otimes_{\mathbb C}\C^2)$~\cite{D,EH}, so
its orbit space is isometric to the cone over the
interval of length $\pi/4$, where the endpoints parametrize
orbits of dimensions~$9$ and~$12$~\cite[p.~436]{HPT}.
Now the enlarged group is orbit equivalent and
its principal isotropy group contains $w$,
so it has no non-zero fixed vectors in the isotropy representation.
This proves that the enlarged group acts asystatically on $M$.

It is now easy to see that
$\Spin7\cdot\SO2$ and $\SU4\cdot\SU2$ do not have proper
closed subgroups acting infinitesimally on $M$. In fact, their slice
representations at~$p$ are irreducible of cohomogeneity two, so,
owing to Lemma~\ref{polar-min}, any subgroup of one of those
groups acting infinitesimally polar on $M$ must be such that
the slice representations at $p$
of the group and its subgroup are orbit-equivalent, but according
to~\cite{D,EH} there cannot exist such a subgroup.

The groups obtained above by adjoining~$w$ to~$\Spin7 \cdot \SO2$
and~$\SU4 \cdot \SU2$, respectively, are well known: they are just the
subgroups $\pi^{-1}(\SOxO72)$ and $\pi^{-1}(\SOxO63)$ of~$\Spin9$,
where $\pi \colon \Spin9 \to \SO9$ is the universal covering map. Indeed,
it is easy to see that the action of~$\psi \in \Aut(\Spin9)$ is given by
conjugation with elements of those subgroups. However, we have preferred
the above alternate point of view, as the Weyl involution appears to be a
useful notion to prove asystaticity in this context, cf.~subsection~\ref{su3su3}.

\subsubsection{$\Spin5\cdot\Spin4$}\label{spin54}
Consider the restriction of the spin representation of $\Spin9$ to the
subgroup $\Spin5\cdot\Spin4$. This representation can be regarded as
$\H^2 \otimes_{\H} (\H \oplus \H)$, where the representation of
$\Spin5\cdot\Spin4 \cong \SP2\cdot(\SP1\times\SP1)$ is given in such
a way that the $\SP2$-factor acts on both copies of $\H^2$ by its
standard representation and where the action of $\SP1 \times \SP1$ on
$\H \oplus \H$ is given componentwise by the standard representation.
This reducible representation is the restriction of
$(\SP2\cdot\SP2,\Q^2\otimes_{\H}\Q^2)$, which is irreducible and of
cohomogeneity~$2$, hence it follows from Lemma~\ref{polar-min} that
neither $\Spin5\cdot\Spin4$ nor any non-trivial connected
closed subgroup thereof can act polarly on~$M$.

\subsubsection{$\SP1\cdot\SP1$}
This maximal connected subgroup of~$\Spin9$ does not act infinitesimally
polar on $M$ since the restriction of the
spin representation of $\Spin9$ on $T_pM=\R^{16}$ to
$\SU2\cdot\SU2$ yields $S^3(\C^2)\otimes_{\mathbb H}\C^2\oplus\C^2\otimes_{\mathbb H} S^3(\C^2)$, namely, a sum
of two polar representations, each with finite principal isotropy groups,
which cannot be polar by~\cite[Th.~4]{D}; in addition, since each irreducible
summand has cohomogeneity two, no subgroup of $\SP1\cdot\SP1$ can act
infinitesimally polar on $M$ (Lemma~\ref{polar-min}).

\subsection{$\SP3\cdot\SP1$ and its subgroups}\label{sp3sp1}

This is a symmetric subgroup of $\F$ so that we have a so-called
\emph{Hermann action} on $M$~\cite{K}. One singular orbit is a totally geodesic
$\Q P^2$ which is also the fixed point set of the $\SP1$-factor. We choose
$p\in\Q P^2$. Then the isotropy subgroup at~$p$ is $\SP1\cdot\SP2\cdot\SP1$,
namely, the group described in subsection~\ref{spin54}. Its
action on $T_pM$ is a sum of two representations of cohomogeneity
one, that is, the isotropy representation of the singular orbit
$T_p(\Q P^2)=\Q\otimes_{\mathbb H}\Q^2$ and the slice
representation $\nu_p(\Q P^2)=\Q^2\otimes_{\mathbb H}\Q$.
In particular the $\SP3\cdot\SP1$-action on $M$ has cohomogeneity one and
it is thus hyperpolar. The principal isotropy group
is obtained from the slice representation: it is the
$\SP1^3$-subgroup given by $q_2=q_4$ in the
diagonal embedding
\[ (q_1,q_2,q_3,q_4)\in\SP1^4\mapsto(q_1,\mathrm{diag}(q_2,q_3),q_4)
\in\SP1\cdot\SP2\cdot\SP1. \]
The action of the principal isotropy group
on $(x,y)\in\Q^2\cong\nu_p(\Q P^2)$ is given by
\[ \left(\begin{array}{cc}q_2&0\\0&q_3\end{array}\right)
\left(\begin{array}{c}x\\y\end{array}\right)q_2^{-1} \]
with fixed point set $\left(\begin{array}{c}\R\\0\end{array}\right)$, and its
action on $(x,y)\in\Q^2\cong T_p(\Q P^2)$ is given by
\[ \left(\begin{array}{cc}q_2&0\\0&q_3\end{array}\right)
\left(\begin{array}{c}x\\y\end{array}\right)q_1^{-1} \]
with trivial fixed point set. In view of Lemma~\ref{crit-asys},
this proves asystaticity of the $\SP3\cdot\SP1$-action on $M$.
It remains to determine whether there are any proper closed subgroups of $\SP3\cdot\SP1$ which act infinitesimally polar.

The maximal connected subgroups of $\SP3$ are
\[
\SP2\cdot\SP1,\quad \U3,\quad \liegr A_1^{35}.
\]
By \cite[Th.~15.1]{Dyn0}, it follows that the maximal connected subgroups of $\SP3\cdot\SP1$ are
\[
\SP2\cdot\SP1\cdot\SP1,\quad \U3\cdot\SP1,\quad \liegr A_1^{35}\cdot\SP1,\quad \SP3\cdot\U1.
\]
\subsubsection{$\SP2\cdot\SP1\cdot\SP1$}
This group is exactly the subgroup $\Spin5\cdot\Spin4$ of $\Spin9$, treated in~\ref{spin54}.

\subsubsection{$\U3\cdot\SP1$}
Regular subgroups, in particular, subgroups of maximal rank in a compact
semisimple Lie group can be studied using Borel-de-Siebenthal theory,
cf.~\cite[Ch.~1, \S3, Thm.~16]{On3}. The maximal
connected subgroups of maximal
rank are obtained by deleting nodes from the extended Dynkin diagram, the
remaining nodes then correspond to a set of simple roots of the subalgebra.
The extended Dynkin diagram of $\F$ is the following.
\[
\begin{xy}
\POS (-10,0) ="z",
\POS (0,0) *\cir<3pt>{} ="a",
\POS (10,0) *\cir<3pt>{} ="b",
\POS (20,0) *\cir<3pt>{} ="c",
\POS (30,0) *\cir<3pt>{} ="d",
\POS (40,0) *\cir<3pt>{} ="e",
\POS (50,0)  ="f",
\POS "a" \ar@{-}^<<<<{} "b",
\POS "b" \ar@{-}^<<<<{} "c",
\POS "c" \ar@{=>}^<<<<{} "d",
\POS "d" \ar@{-}^<<<<{} "e",
\POS "z" \ar@{}^<<<<{} "a",
\end{xy}
\]
The subgroup $\SP3\cdot\SP1$ is obtained by deleting the second node from the left, the subgroup $\SU3\cdot\SU3$ by deleting the middle node. It follows from~\cite[\S3,~Thm.~16]{On3} that a system of simple roots of the subgroup $\U3\cdot\SP1 \subset \SP3\cdot\SP1$ is obtained by deleting the second and middle node from the above diagram, which shows that $\U3\cdot\SP1$ is also contained in $\SU3\cdot\SU3$. The latter group will be treated in subsection~\ref{su3su3}.

\subsubsection{$\liegr A_1^{35}\cdot\SP1$}
Suppose this group or a proper closed subgroup of it acts infinitesimally
polar on $M$. Then this action restricts to an infinitesimally polar
action of a rank one group
on the fixed point set $\Q P^2$ of the $\SP1$-factor. Due to
Theorem~\ref{main}(b), which has already been proved in
section~\ref{classical}, such an action is in fact polar.
However the section, as a totally geodesic submanifold of $\Q P^2$,
can have dimension at most $4$. On the other hand, the cohomogeneity of the action on $\liegr A_1^{35}\cdot\SP1$
on $\Q P^2$ is (since the $\SP1$-factor acts trivially) at least~$5$, a contradiction.

\subsubsection{$\SP3\cdot\U1$}
We assert that $\SP3\cdot\U1$ and $\SP3$ act on $M$ with the same orbits as
$\SP3\cdot\SP1$; it is enough to prove the second assertion. Indeed, consider
the fixed point set $\Q P^2$ of the $\SP1$-factor and fix a
basepoint~$p\in\Q P^2$. Of course $\SP3$ acts transitively on $\Q P^2$ with
isotropy group $\SP1\cdot\SP2$, and the description of the slice
representation given in subsection~\ref{sp3sp1} shows that $\SP1\cdot\SP2$
acts with cohomogeneity one on $\nu_p(\Q P^2)$, proving the assertion.
The other proper closed subgroups of $\SP3\cdot\U1$
have already been considered above as
subgroups of the other maximal connected subgroups of $\SP3\cdot \SP1$.

\subsection{$\SU3\cdot\SU3$ and its subgroups}\label{su3su3}

One $\SU3$-factor, say the second, is contained in $\G$
and its fixed point set in $M$ is a totally
geodesic $\C P^2$; the other $\SU3$-factor acts
transitively on this $\C P^2$.
We fix a basepoint $p\in\C P^2$. Put $G=\SU3\cdot\SU3$.
The isotropy group $G_p$
is $\mathsf{S}(\U1\times\U2)\cdot\SU3\cong\U2\times\SU3$, and its slice
representation is $\C^2\otimes_{\mathbb C}\C^3$. From here we deduce
that the principal isotropy group of the $G$-action on $M$
is the maximal torus ${\sf T}^2$ of the diagonal
$\SU3$-subgroup of $G$.
Its fixed point set in the isotropy representation of a principal orbit is the Killing orthogonal complement
of the Lie algebra of ${\sf T}^2$ in the Lie algebra of the
maximal torus of $G$. We enlarge $G$ by adjoining $w\in\Spin9$
where $\psi=\mathrm{Ad}_w$ is a Weyl involution of~$\F$
relative to a maximal torus contained in $\SU3\cdot\SU3\cap\Spin9=
\mathsf{S}(\U1\times\U2)\cdot\SU3$.
To see that the enlarged group has the same
orbits, note first that $\psi$ fixes $p$ so it
preserves the singular $G$-orbit $\C P^2$.
Moreover it induces
an isometry of the orbit space of the slice representation
at~$p$, which must be trivial.
Indeed that orbit space is isometric to the cone over
the interval of length~$\pi/4$, where the endpoints parametrize
orbits of dimensions~$7$ and~$8$~\cite[p.~436]{HPT}.
Since $\exp_p(\nu_p(\C P^2))$
meets every $\SU3\cdot\SU3$-orbit, this shows that $\psi$
preserves the $G$-orbits in $M$.
Now the enlarged group has the same orbits and its principal
isotropy group contains $w$. Since $\psi$ acts as minus
identity on the maximal torus of $G$, this proves that
the enlarged groups acts asystatically.

Now assume a closed subgroup~$H$ of~$G$ acts infinitesimally polar on~$M$.
We will apply a similar argument as in \cite[p.~454]{K3}.
It follows for all $g \in G$ that also $gHg^{-1}$ acts infinitesimally polar on~$M$.
The slice representation of the $G$-action at~$p$ is irreducible
of cohomogeneity two and it follows from Lemma~\ref{polar-min}
and~\cite{D,EH} that $G_p \cap gHg^{-1}$ is either finite or contains
$\SU2 \times \SU3$.
In the former case, it follows that $\dim H \le 4$,
in the latter case $\dim H \ge 11$.
Since these two conditions are mutually exclusive we have that
$G_p \cap gHg^{-1}$ either is finite for all $g \in G$ or contains
$\SU2 \times \SU3$ for all $g \in G$.
If $G_p \cap gHg^{-1}$ is finite for all $g \in G$, it follows that $H$ is finite,
since $G_p$ contains a maximal torus of~$G$.
Otherwise, it follows that $H$ contains $g(\SU2 \times \SU3)g^{-1}$ for all $g \in G$;
hence $H$ contains the normal subgroup generated by the union of $g\, \SU2\, g^{-1} \times \SU3$, $g \in G$. Since $\SU3$ is a simple Lie group, it follows that $H=G$.

\subsection{$\G\cdot\liegr A_1^8$ and its subgroups}\label{g2so3}

We will prove that this group acts asystatically.
In order to describe its action on the Cayley projective plane, we
will use the models of projective spaces
over the normed division algebras $\K$ ($\K=\R$, $\C$, $\Q$, $\Ca$)
given by idempotent Hermitian
matrices of trace $1$~\cite{CR}. Let $\mathcal J$
be the Jordan algebra of Hermitian $3 \times 3$-matrices with entries
in $\Ca$, where the multiplication is defined by
$x \circ y = \frac12(xy+yx)$.
Let $\mathrm{Herm}_\varepsilon(3,\K)$ be the subspace of $\mathcal J$
consisting of Hermitian matrices with entries in $\K$ and
trace~$\varepsilon \in \R$. Then $M=\Ca P^2$ is embedded in
$\mathcal J$ as the smooth real algebraic subvariety
$V=\{ x \in \mathrm{Herm}_1(3,\Ca) \mid x^2=x \}$ of $\mathcal J$.
It is well known that the automorphism group of the algebra $\mathcal J$
is the compact Lie group $\F$. Furthermore, the action of this group leaves $V$ invariant and acts on it as the isometry group of the Riemannian symmetric space $\Ca P^2$. We also identify $\Ca$ with $\Q\times\Q$ via the Cayley-Dickson
process, so that the multiplication in~$\Ca$ is given by
\begin{equation}\label{Def0}
(a,b)(c,d)=(ac-\bar db,da+b \bar c),
\end{equation}
and recall that $\G$ is the automorphism group of~$\Ca$.
Using~(\ref{Def0}), one sees that the maps
\[
\alpha_{p,q} \colon \Ca \to \Ca, \quad (x,y) \mapsto (pxp^{-1},qyp
^{-1}),
\]
where $p$, $q\in\SP1$, comprise the maximal subgroup $\SO4$ of $\G$.

Now the action
of $G:=\SO3\times\G$ on $V$ is simple to describe:
\[ (A,\alpha)\cdot x = A(\alpha(x_{ij}))A^t \]
where $A\in\SO3$, $\alpha\in\G$, $x=(x_{ij})\in V$.
Since both factors are centerless groups, this indeed
defines an effective action of $\SO3\times\G$ on $V$.

Since the fixed point set of $\G$ in $\Ca$ is $\R$, we see that
the fixed point set~$V^{{\sf G}_2}=V\cap\mathrm{Herm}_1(3,\R)=\R P^2$
and that $\SO3$ acts transitively on that set.
For any $x\in V$, we have
\begin{equation}\label{1}
 T_xV=\{y\in\mathrm{Herm}_0(3,\Ca)\;|\;xy+yx=y\}.
\end{equation}
We fix a basepoint
\[ p=\left(\begin{array}{ccc}1&0&0\\0&0&0\\0&0&0\end{array}\right)\in V. \]
It is immediate that $G_p=\OG2\times\G$, where $\OG2$ is embedded into $\SO3$
via
\[ B\mapsto\left(\begin{array}{cc}\det B&0\\0&B\end{array}\right) \]
and a simple calculation using~(\ref{1}) shows that
the normal and tangent spaces to $V^{{\sf G}_2}=\R P^2$ in $V=\Ca P^2$ are:
\[  \nu_p(\R P^2)=\left\{\,\left(\begin{array}{ccc}0&a&b\\\bar a&0&0\\\bar b&0&0\end{array}\right)\;\Big|\;a,\; b\in\Im\Ca\,\right\} \]
and
\[ T_p(\R P^2)=\left\{\,\left(\begin{array}{ccc}0&a&b\\\bar a &0&0\\\bar b&0&0\end{array}\right)\;\Big|\;a,\; b\in\R\,\right\}. \]
Therefore the slice representation $(G_p,\nu_p(\R P^2))$ is
$(\OG2\times\G,\R^2\otimes_{\R}\R^7)$, where $\Im\Ca\cong\R^7$.
This representation is orbit equivalent to
$(\SO2\times\SO7,\R^2\otimes_{\R}\R^7)$, polar, and
of cohomogeneity~$2$~\cite{EH}. The subspace of $\nu_p(\R P^2)$ spanned by
\[ \left(\begin{array}{ccc}0&i&0\\-i&0&0\\0&0&0\end{array}\right)\quad\mbox{and}\quad
\left(\begin{array}{ccc}0&0&j\\0&0&0\\-j&0&0\end{array}\right) \]
is a section, from which we find a principal isotropy group
\[ H\cong(\Z_2)^2\times\SP1; \]
here $\SP1$ is the normal subgroup of
$\SO4\subset\G$ that acts trivially on $\Q\cong\Q\times\{0\}\subset
\Ca$, namely, generated by $\alpha_{1,q}$ for $q\in\SP1$,
and $(\Z_2)^2$ is embedded diagonally in $\SO3\times\G$,
with generators
\[ h_1=\left(\left(\begin{array}{ccc}-1&0&0\\0&-1&0\\0&0&1\end{array}\right),
\alpha_{i,i}\right) \]
and
\[ h_2=\left(\left(\begin{array}{ccc}-1&0&0\\0&1&0\\0&0&-1\end{array}\right),
\alpha_{j,j}\right). \]
The fixed point set of $H$ in $V$ can be obtained as follows.
First note that $V^{H^0}=V\cap\mathrm{Herm}_1(3,\Q)=\Q P^2$.
A direct calculation now shows that $V^H=(V^{H^0})^H$
consists of matrices
\[ \left(\begin{array}{ccc}a^2&-abi&-acj\\abi&b^2&-bck\\acj&bck&c^2\end{array}\right) \]
where $a$, $b$, $c\in\R$ and $a^2+b^2+c^2=1$. This is a $2$-dimensional
submanifold and clearly a copy of $\R P^2$. This proves that the
$G$-action on $M$ is asystatic; in particular, it is polar.

We proceed to compute the generalized Weyl group and the orbit space.
For computational ease, we apply the following result, which is also
of independent interest. It is a modification of the Luna-Richardson-Straume
reduction, cf.~\cite[Ex.~2.2.1]{Gor4} or~\cite[\S2.6]{G-L3},
where, instead of a principal isotropy group,
only its identity component is used.
See~\cite[Sect.~2]{GOT} for the definition of a generalized section.

\begin{lem}\label{gen-sec}
Assume a Lie group $G$ acts properly and isometrically on
a connected complete Riemannian manifold~$M$. Consider the identity component
$H^0$ of a fixed principal isotropy group~$H$. There is a component $\Sigma$
of the fixed point set $M^{H^0}$ which contains a point with isotropy
group~$H$.
Then $\Sigma$ is a generalized section for the $G$-action
on~$M$. In particular, there is an isometry of orbit spaces
$M/G=\Sigma/\mathcal W$, where the generalized Weyl group is defined by
$\mathcal W = N_G(\Sigma) / Z_G(\Sigma)$.
Moreover, $N_G(\Sigma)$ is an open and closed subgroup of
$N_G(H^0)$ and $Z_G(\Sigma)$ is an open and closed subgroup of~$H$.
\end{lem}

\begin{proof}
One can always pick a component $\Sigma$ of $M^{H^0}$
as in the statement.
It is clear that $\Sigma$ is a connected closed totally geodesic submanifold
of $M$ such that its tangent space $T_p\Sigma$ contains
the normal space $\nu_p(G \cdot p)$ for every $G$-regular point $p\in\Sigma$,
and which thus meets all $G$-orbits.
To see it is a generalized section, it remains only to
prove that if $p$, $q=g \cdot p\in\Sigma$ are $G$-regular points for
some $g\in G$ then $g \cdot \Sigma=\Sigma$.

Assume that $p$, $q \in \Sigma$ are regular points such that $q = g \cdot p$.
Then we have $(G_p)^0=H^0=(G_q)^0$ and $G_q=gG_pg^{-1}$. This implies
$H^0 = (G_q)^0=(g G_p g^{-1})^0 = g H^0 g^{-1}$, showing that $g\in N_G(H^0)$.
Hence $N_G(\Sigma)$ is a closed subgroup of $N_G(H^0)$.
Since $g$ normalizes $H^0$, we have $g \cdot M^{H^0}=M^{H^0}$ and
it follows that $g$ maps the connected component of $M^{H^0}$
containing~$p$ onto that component containing~$q$ (the same),
that is, $g \cdot \Sigma=\Sigma$.
This completes the proof that $\Sigma$ is a generalized section.
The assertion about the orbit space now follows
(compare~\cite[Thm.~2.1.1]{Mag2} and~\cite[Prop.~2.2.1]{Gor4}).
It is also clear that $N_G(\Sigma)$ contains the identity component
of $N_G(H^0)$, and hence $N_G(\Sigma)$ is open in that group.

Finally, it is obvious that $H^0 \subset Z_G(\Sigma)$.
Let $g \in Z_G(\Sigma)$ and assume $p \in \Sigma$ has $G_p=H$.
Then $g \in H$.
\end{proof}

It is not hard to see
that the normalizer $N_{{\sf G}_2}(H^0)=\SO4$, so
$N_G(H^0)=\SO3\times\SO4$.
Since this normalizer is connected, it follows from
Lemma~\ref{gen-sec} that $N_G(\Sigma)=N_G(H^0)$.
It also follows from the lemma that $Z_G(\Sigma) = H^0$, since  the elements $h_1$ and $h_2$ defined above act nontrivially on $M^{H^0}$.
Now $\mathcal W = N_G(H^0)/H^0=\SO3\times\SO3$ acts on $V^{H^0}=\Q P^2$, and the map
\[ (x\in\Q^3, ||x||=1) \mapsto xx^*\in\mathrm{Herm}_1(3,\Q) \]
is $\SP3$-equivariant and induces the standard embedding of $\Q P^2$
into $\mathrm{Herm}_1(3,\Q)$, so we can see the $\mathcal W$-action in homogeneous
coordinates as
$(A,\pm q)\in\SO3\times\SP1/\Z_2=\SO3\times\SO3$ acting on
$\left[\begin{array}{c}x_1\\x_2\\x_3\end{array}\right]\in\Q P^2$
by
$A\left[\begin{array}{c}qx_1q^{-1}\\qx_2q^{-1}\\qx_3q^{-1}\end{array}
\right]=A\left[\begin{array}{c}qx_1\\qx_2\\qx_3\end{array}\right]
\in\Q P^2$.
As in section~\ref{classical}, this action canonically
lifts to a representation of $\SO3\times\SP1\times\SP1$ on~$\Q^3$
of cohomogeneity $3$, where the last $\SP1$-factor
acts by right multiplying each coordinate by a unit quaternion.
Finally, this representation is, up to a $\Z_2$-kernel,
equivalent to~$(\SO3\times\SO4,\R^3\otimes_{\R}\R^4)$
and hence polar.
The Weyl group
for the action of $\SO3\times\SO4$ on $S^{11}$
is of type ${\liegr C}_3$ acting on $S^2$, so the
generalized Weyl group
for the action on $\Q P^2$ (and for the action on $M$)
is of type ${\liegr C}_3/\Z_2$ acting on $\R P^2$. Note that this argument
also proves polarity of the $G$-action on $M$.

An analogous argument as in the last paragraph of~\ref{su3su3},
applied to the slice representation of $G_p = \OG2 \times \G$ at~$p$,
shows that any closed subgroup of $G$ acting infinitesimally polar on~$M$
equals~$G$.

\begin{rmk}
It follows rather easily from $V\cap\mathrm{Herm}_1(3,\K)\cong\K P^2$
that the fixed point sets of the subgroups in the chain
$\Spin9\supset \G \supset \SU3 \supset \SP1 \supset \{1\}$, where
$\SP1$ is an index~$1$ subgroup of $\F$, yield a
chain of totally geodesic submanifolds
$\{\mathrm{pt}\}\subset\R P^2\subset \C P^2
\subset \Q P^2 \subset \Ca P^2$.
\end{rmk}

\subsection{Rank one subgroups}\label{rk1}

The action of a subgroup $G$
of rank one of $\F$ cannot be infinitesimally polar.
In fact its maximal torus (a circle subgroup)
is conjugate to a subgroup of $\Spin9$ so it
has a fixed point, say $q$. The normal space to the $G$-orbit
through $q$ in $M$ has dimension at least $14$. The components of
the fixed point set in $M$
of any non-trivial subgroup of $\F$ are totally geodesic submanifolds
of dimension at most $8$. It follows that the isotropy group at $q$
acts without fixed directions
on a subspace of the normal space of dimension at least $6$,
but there are no polar representations of $\U1$ or $\SU2$
without fixed directions in that dimension.

\section{Addendum}

The following construction yields closed subgroups of $\F$
that act on $\Ca P^2$ with a totally geodesic singular orbit.
The list includes all polar actions on $\Ca P^2$.
Let~$P$ be a connected closed totally geodesic submanifold of $\Ca P^2$.
Let $N(P)$ and $Z(P)$ denote the identity components of the subgroups
of $\F$ consisting of elements that preserve $P$, resp., fix $P$ pointwise.
Note that $Z(P)$ is a normal subgroup of $N(P)$.
The possibilities for $P$ are well known~\cite{Wolf3}.
\begin{table}
\begin{minipage}{\textwidth}
\[\begin{array}{|c|c|c|c|c|c|}
\hline
P & Z(P) & N(P) & \textsl{Slice repr} & \textsl{Cohom} & \textsl{Polar?} \\

\hline

\{\mathrm{pt}\} & \Spin9 & \Spin9 & (\Spin9,\R^{16}) & 1 & \textrm{yes} \\ 

 \R P^2 & \G & \G \times \SO3 & (\G\times\OG2,\Im(\Ca) \otimes_\R \R^2) & 2 & \textrm{yes} \\ 

 \C P^2 & \SU3 & \SU3 \cdot \SU3 & (\SU3\times\U2,\C^3 \otimes_\C \C^2) & 2 & \textrm{yes} \\ 

 \H P^2 & \SP1 & \SP1 \cdot \SP3 & (\SP1\times\SP2,\H^1 \otimes_\H \H^2) & 1 & \textrm{yes} \\ 

 S^1 & \Spin7 & \Spin7 \cdot \SO2 & (\Spin7,\R^7 \oplus \R^8) & 2 & \textrm{yes} \\ 

 S^2 & \SU4 & \Spin6\cdot \Spin3 & (\U4,\R^6 \oplus \C^4) & 2 & \textrm{yes} \\ 

 S^3 & \SP2 & \Spin5 \cdot \Spin4 & (\SP1\cdot\SP2,\R^5\oplus\Q\otimes_{\mathbb H}\Q^2) &  3 &  \textrm{no} \\ 

 S^4 & \SP1 \cdot \SP1 & \Spin4 \cdot \Spin5 & (\SP1^4,\H \oplus \H \oplus \H) & 3 & \textrm{no} \\ 

 S^5 & \SP1 & \Spin3  \cdot \Spin6 & (\SP1\cdot\SP2,\R^3 \oplus \C^4) & 2 & \textrm{yes} \\ 

 S^6 & \SO2 & \SO2 \cdot \Spin7 & (\U4, \R^2 \oplus \C^4) & 2 & \textrm{yes} \\ 

 S^7 & \{1\} & \Spin8 & (\Spin7,\R\oplus\R^8) & 2 & \textrm{yes} \\ 

 S^8 & \{1\} & \Spin9 & (\Spin8,\R^8) & 1 & \textrm{yes} \\ \hline

\end{array}\]
\begin{center}
\hfill\\
\textsc{Table~4}\hfill\\
\hfill\\
\end{center}
\end{minipage}
\end{table}
In Table~4 we
indicate the (almost effectivized)
slice representation for the action of $N(P)$ on $\Ca P^2$ at a point $p$ in $P$, the cohomogeneity, and whether that action is polar or not.
Note $P=S^4$ is the only one case in which the action
of $N(P)$ on $\Ca P^2$ is not polar, see subsection~\ref{spin54}, but its
slice representation at $p$ is.
In fact, the slice representation is equivalent to the action of~$\SP1^4$ on~$\H \oplus \H \oplus \H$ given by
$(a,b,c,d) \cdot (x,y,z) = (axb^{-1},axc^{-1},bxd^{-1}).$\hfill\\

\par
{\small
\noindent\textit{Complements and corrections.} The authors take this
opportunity to make corrections to some of their previous papers.

In~\cite{B-Go}, the proof that polar actions on the Cayley projective
plane are taut was done case-by-case using a reduction argument,
but the group $\SO3\times\G$ was not considered. It is readily seen that the same method by reduction can be applied to this group.
Moreover a short proof of the main result of that paper, that does not
rely on classification results, is obtained from~\cite[Theorem~3.20]{Wiesen1}.

In~\cite[Sec.~11]{K4}, a classification of polar actions on the
Cayley hyperbolic plane was given under the hypothesis that the group acting
is a reductive algebraic subgroup of the isometry group. However, the results
of the present article show that, in addition to the actions given there,
also the group $\SO{2,1} \cdot \G$ acts polarly on $\Ca H^2$ and with an
orbit which is a totally geodesic~$\R H^2$.
Since there are no totally geodesic
orbits of $\SO3 \times \G$ on~$\Ca P^2$ other than the $\R P^2$ on which
the~$\SO3$-factor acts as the isometry group and which is fixed by
the~$\G$-factor, it follows that the complete list of all connected
reductive algebraic subgroups of the noncompact form of $\F$ which act
polarly on $\Ca H^2$ is given by the the actions in \cite[Thm~11.1]{K4}
together with the action of $\SO{2,1} \cdot \G$.
}

\providecommand{\bysame}{\leavevmode\hbox to3em{\hrulefill}\thinspace}
\providecommand{\MR}{\relax\ifhmode\unskip\space\fi MR }
\providecommand{\MRhref}[2]{%
  \href{http://www.ams.org/mathscinet-getitem?mr=#1}{#2}
}
\providecommand{\href}[2]{#2}


\end{document}